\documentclass[12pt,centertags,oneside]{amsart}
\usepackage{amsmath,amstext,amsthm,amscd,typearea}
\usepackage{amssymb}
\usepackage{amsfonts}
\usepackage{soul,cancel}
\setstcolor{red}
\usepackage[backref=page]{hyperref}

\usepackage{tikz}\usetikzlibrary{decorations.markings,matrix,arrows}
\usepackage{tikz-cd}

\usepackage{amscd,amsxtra,calc,mathtools}
\usepackage{cmmib57}
\usepackage{url}
\usepackage{amscd}
\usepackage{color}
\usepackage[normalem]{ulem}
\usepackage{braket}

\usepackage[a4paper,width=16.2cm,top=3cm,bottom=3cm]{geometry}

\numberwithin{equation}{section}
\theoremstyle{plain}
\newtheorem{thm}{Theorem}[section]
\newtheorem{theorem}[thm]{Theorem}

\newtheorem{lemma}[thm]{Lemma}
\newtheorem{lem}[thm]{Lemma}

\newtheorem{proposition}[thm]{Proposition}

\theoremstyle{definition}
\newtheorem{remark}[thm]{Remark}

\newtheorem{claim}[thm]{Claim}

\newtheorem{s:examples}[thm]{s:examples}
\newtheorem{conjecture}[thm]{Conjecture}

\newtheorem{question}[thm]{Question}

\numberwithin{equation}{section}

\newcommand{\ga}[2]{\begin{gather}\label{#1}#2 \end{gather}}

\newcommand{\ess}{{\rm ess}}

\newcommand{\C}{{\mathbb C}}

\newcommand{\R}{{\mathbb R}}

\newcommand{\Z}{{\mathbb Z}}

\newcommand{\Vect}{\text{\sf Vect}}

\newcommand{\Aut}{{\rm Aut\hspace{.1ex}}}

\newcommand{\Nef}{{\rm Nef\hspace{.1ex}}}

\newcommand{\GL}{{\rm GL\hspace{.1ex}}}

\newcommand{\Ker}{{\rm Ker\hspace{.1ex}}}

\newcommand{\Kc}{\mathcal{K}}
\newcommand{\Nc}{\mathcal{N}}

\newcommand{\ssec}{\subsection}

\newcommand{\sssec}{\subsubsection}

\newcommand{\ol}{\overline}

\newcommand{\vast}{\bBigg@{4}}
\newcommand{\Vast}{\bBigg@{5}}

\newcommand{\cK}{\mathcal{K}}

\newcommand{\cN}{\mathcal{N}}

\newcommand{\gL}{\Lambda}

\renewcommand{\ga}{\alpha}
\newcommand{\gb}{\beta}

\newcommand{\gk}{\kappa}

\newcommand{\go}{\omega}

\newcommand{\Id}{\mathrm{Id}}

\newcommand{\torsion}{\mathrm{torsion}}

\newcommand{\bss}{\backslash}

\newcommand{\cnec}{\mathrel{:=}}
\newcommand{\cupp}{\mathbin{\smile}}

\renewcommand{\(}{\left(}
\renewcommand{\)}{\right)}

\newcommand{\cto}{\circlearrowleft}

\newcommand{\hto}{\hookrightarrow}

\title[Automorphisms of compact K\"ahler manifolds]
{Zero entropy automorphisms of compact K\"ahler manifolds and dynamical filtrations}

\dedicatory{In memory of Jean-Pierre Demailly}

\author{Tien-Cuong Dinh}
\address{Department of Mathematics, National University
of Singapore,
Singapore}
\email{matdtc@nus.edu.sg}

\author{Hsueh-Yung Lin}
\address{Department of Mathematics, National Taiwan University,
Taipei, Taiwan}
\email{hsuehyunglin@ntu.edu.tw}

\author{Keiji Oguiso}
\address{
Graduate School of Mathematical Sciences, the University of Tokyo,
Japan; National Center for Theoretical Sciences, 
National Taiwan University,
Taipei, Taiwan
}
\email{oguiso@g.ecc.u-tokyo.ac.jp}

\author{De-Qi Zhang}
\address{Department of Mathematics, National University
of Singapore,
Singapore}
\email{matzdq@nus.edu.sg}

\medskip

\date{}

\begin{document}

\begin{abstract}
	We study
	zero entropy automorphisms of
	a compact K\"ahler manifold $X$.
	Our goal is to bring to light
	some new structures of the action on the cohomology of $X$,
	in terms of the so-called dynamical filtrations on $H^{1,1}(X,\R)$.
	Based on these filtrations,
	we obtain the first general
	upper bound on the
	polynomial growth of the iterations
	$(g^m)^* \cto H^2(X,\C)$ where $g$ is a zero entropy automorphism,
	in terms of $\dim X$ only.
	
We also give an upper bound for the (essential) derived length $\ell_{\rm ess}(G, X)$ for every zero entropy subgroup $G$, again in terms of the dimension of $X$ only. We propose
a conjectural upper bound for the essential nilpotency class
$c_\ess(G,X)$ of a zero entropy subgroup $G$.

Finally, we construct examples showing
	that our upper bound of the polynomial growth
	(as well as the conjectural upper bound of $c_\ess(G,X)$) are
	optimal.
\end{abstract}



\maketitle

\tableofcontents

\section{Introduction } \label{s:intro}

Let $X$ be a compact K\"ahler manifold and let $\Aut(X)$ be the group of holomorphic automorphisms of $X$.
Recall that for every $g \in \Aut(X)$, according to theorems of Gromov and Yomdin, the topological entropy of $g$ is equal to the logarithm of the spectral radius of its action on the cohomology ring of $X$, see~\cite{Dinh,Gromov,Yomdin}.
Groups with positive entropy elements have been intensively studied
during the last decades using techniques and ideas
from complex dynamics and algebraic geometry,
see \emph{e.g.}~\cite{CantatICM, DinhICM, OguisoICM} and the references therein.
Despite these developments, much less is known about zero entropy elements (see \emph{e.g.}~\cite[\S 3]{CantatICM}), which will be the main focus of this paper.

One of the aims of this paper is to excavate
some hidden structures of
zero entropy group actions on the cohomology of $X$,
in terms of various \emph{dynamical filtrations} on $H^{1,1}(X,\R)$.
We defer the discussion on these dynamical filtrations to the later part of the introduction,
and start with the consequences.

\ssec{Polynomial growths of pullback actions}
\hfill

As a consequence of Gromov--Yomdin's theorem,
an automorphism $g \in \Aut(X)$ of a compact K\"ahler manifold $X$ has zero entropy
if and only if
the pullback
$g^* : H^\bullet(X,\C) \cto$ is quasi-unipotent
(see e.g. Proposition~\ref{p:main_finind}(1)).
In this case, the iterations $(g^m)^* : H^\bullet(X,\C) \cto$ of $g^*$
have polynomial growth 
when $m$ tends to infinity.
The study of the polynomial growth of $(g^m)^* : H^2(X,\C) \cto$
was initiated by Lo Bianco~\cite{Lo}:
he proved that the growth satisfies
$$\|(g^m)^*: H^2(X,\C)\circlearrowleft\| =_{m \to \infty} O\big(m^{4}\big).$$
when $\dim X = 3$.
With the help of dynamical filtrations,
we generalize Lo Bianco's theorem to
compact K\"ahler manifolds of arbitrary dimension.

\begin{theorem}\label{t:main-1map}
	Let $X$ be a compact K\"ahler manifold of dimension $n \ge 1$. Let $g$ be any automorphism of zero entropy of $X$.
	Then the action of $g^m$ on $H^{1,1}(X,\C)$
	satisfies
	$$\|(g^m)^*: H^{1,1}(X,\C)\circlearrowleft\| =O\big(m^{2(n-1)}\big)$$
	when $m$ tends to infinity.
	This estimation is optimal in terms of $n = \dim X$.
	
	More generally, the action of $g^m$ on each factor $H^{p,q}(X,\C)$ of the Hodge decomposition with $0\leq p,q\leq n$, satisfies
	$$\|(g^m)^*: H^{p,q}(X,\C)\circlearrowleft\| =O\big(m^{(p'+q')(n-1)}\big)$$
	when $m$ tends to infinity; here $p':=\min(p,n-p)$ and $q':=\min(q,n-q)$.
\end{theorem}

According to~\cite[\S 3.2]{CantatICM},
Theorem~\ref{t:main-1map} appears to be the first general result
on the growth of actions on $H^2(X,\C)$ induced by zero entropy automorphisms. 	
Note that
	$$\|(g^m)^*:H^{p,q}(X,\R)\circlearrowleft\|=O(m^{k-1})$$
	if and only if the Jordan blocks of
	$g^*:H^{p,q}(X,\R)\circlearrowleft$ in its Jordan form are of size $k\times k$ or smaller.

\ssec{Derived lengths of zero entropy subgroups}
\hfill

Given a group $H$, its $p$-th {\it derived subgroups} $H^{(p)}$ are defined inductively by
$$H^{(0)} := H\quad \text{and} \quad H^{(i+1)} := [H^{(i)}, H^{(i)}]\, ,$$
where $[H^{(i)}, H^{(i)}]$ is the subgroup of $H$
generated by $[g,h] = ghg^{-1}h^{-1}$ ($g,h \in H^{(i)}$).
By definition, $H^{(p)} = \{1\}$ for some integer $p \ge 0$ if and only if $H$ is solvable. We call the minimum of such $p$ the {\it derived length} of $H$ (when $H$ is solvable) and denote it
by $\ell (H)$.
If $H$ is not solvable, we set $\ell(H) := \infty$.

A subgroup $H \subset \GL(N,\R)$ is said to be {\it unipotent} 
if $1$ is the only eigenvalue of every element of $H$.
It is known that a subgroup $H$ of $\GL(N,\R)$ is unipotent
if and only if $H$ is conjugate to a subgroup 
consisting of upper triangular matrices 
with all entries on the diagonal being $1$~\cite[\S 17.5]{Hum}.
Note that in this discussion, we can replace $\R$ by $\C$ as we have the natural inclusions $\GL(N,\R)\subseteq \GL(N,\C)\subseteq \GL(2N,\R)$.

Since unipotent subgroups are nilpotent and hence solvable, every unipotent subgroup $H$ has derived length
$\ell(H)  < \infty$.

We say that a subgroup $G \subset \Aut(X)$ has zero entropy,
and call it a \emph{zero entropy subgroup},
if all its elements have zero entropy.
For every subgroup $G \subset \Aut(X)$, we define
$$G_0 := G \, \cap \, \Aut^0(X),$$
where $\Aut^0(X)$ is the identity component of $\Aut(X)$.
Note that $\Aut^0(X)$ is a connected complex Lie group of finite dimension by Bochner--Montgomery's theorem \cite{Bochner}.
For every zero entropy subgroup $G \subset \Aut(X)$
one can show that
there exists a finite index subgroup $G' \subset G$ such that
the action of $G'/G'_0$ on $H^{p,p}(X,\R)$
 realizes $G'/G'_0$ as a unipotent subgroup of $\GL(H^{p,p}(X,\R))$ 
 for every $1 \le p \le \dim X -1$.
Moreover,
the invariant
$$\ell_{{\rm ess}} (G, X) := \ell(G'/G'_0),$$
does not depend on the choice of $G'$ (see Proposition~\ref{p:main_finind}).
We call $\ell_{{\rm ess}} (G, X)$ the
{\it essential derived length} of the subgroup $G \subset \Aut(X)$.

We will prove the following in \S\ref{s:length}, based on dynamical filtrations.

\begin{theorem} \label{t:main_1}
Let $X$ be a compact K\"ahler manifold of dimension $n \ge 1$. Then
every subgroup $G \subset \Aut(X)$ of zero entropy satisfies
$$\ell_{\ess}(G, X) \le n - 1.$$
\end{theorem}

\ssec{Dynamical filtrations}\label{sect_filtr}
\hfill

In this subsection,
we present one example of dynamical filtration,
the one which
underlies the estimate in Theorem~\ref{t:main-1map}.

Let $X$ be a compact K\"ahler manifold of dimension $n \ge 1$.
For simplicity, if $L$ and $M$ are two cohomology classes, we denote by $L\cdot M$ or $LM$ their cup-product.
We also identify $H^0(X,\R)$ and $H^{2n}(X,\R)$ with $\R$ in the canonical way, where $n = \dim X$.
So classes in these groups are identified with real numbers.
The {\it nef cone} $\Nef(X) \subset H^{1,1}(X,\R)$ is defined as the closure of the K\"ahler cone.

Let $G \subset \Aut(X)$ be a subgroup such that
the induced action
$G \cto H^{1,1}(X)$ is unipotent.
For instance, $G$ can be the cyclic group generated by some finite power of a zero entropy automorphism.
Based on a Lie--Kolchin type theorem~\cite[Theorem 1.1]{KOZ},
there exists a sequence
$$M_1,\ldots,M_n \in H^{1,1}(X,\R)$$
such that each $L_i \cnec M_1 \cdots M_i \ne 0 \in H^{i,i}(X,\R)$
(with $L_0 \cnec 1 \in H^{0}(X,\R)$)
is $G$-invariant and satisfies
$$L_i \in \ol{L_{i-1} \cdot \Nef(X)}.$$
Following~\cite{Zhang3} and~\cite{Dinh},
we call $M_1,\ldots,M_n \in H^{1,1}(X,\R)$
a {\it quasi-nef sequence} for $G \cto X$.

For $\gb \in H^{i,i}(X,\R)$,
 we write $\gb \equiv 0$ if $\gb \cdot H^{1,1}(X,\R)^{n-i} = 0$.
Define
$$F_i \cnec \Set{\alpha \in H^{1,1}(X,\R) | L_i \alpha \equiv 0}$$
and
$$F'_i \cnec \Set{ \alpha \in F_i | L_{i-1} \alpha^2 \equiv 0}.$$
Based on the mixed Hodge--Riemann relations~\cite{DN, Gromov1}
and the positivity of nef classes,
we will get:

\begin{thm}\label{thm-filtintro}
	The subsets $F_i$ and $F'_i$ form a filtration
	$$0 = F_0 \subset F_1' \subset F_1 \subset \cdots
	\subset F'_{n-1} \subset F_{n-1} \subset F'_n = H^{1,1}(X,\R)$$
	of $G$-stable vector spaces, where $n = \dim X$.
	Moreover,
	\begin{enumerate}
		\item $\dim F'_i/F_{i-1} \le 1$; and the following assertions are equivalent:
		\begin{enumerate}
			\item $\dim F'_i/F_{i-1} = 1$;
			\item $F'_i = F_{i-1} \oplus \R \cdot M_i$;
			\item $L_{i-1} M^2_i = 0$.
		\end{enumerate}
	
		\item
		For every $g \in G$, there exists
		a unique strictly decreasing sequence of integers
		$$n-1\geq s_1>\cdots > s_r \ge  1$$
		such that for every K\"ahler class $\go \in H^{1,1}(X,\R)$,
		we have
		$$(g^* - \Id)^{2j-1}(\go) \in F_{s_j}\setminus F_{s_j}' \ \ \ \ \text{ and } \ \ \ \
		(g^* - \Id)^{2j}(\go)\in F_{s_j}'\setminus F_{s_j-1}$$
		for $1\leq j\leq r$, and $(g^* - \Id)^{2r+1}(\go)=0$.
		
	\end{enumerate}
\end{thm}

The upper bound estimate in
Theorem~\ref{t:main-1map}
for $(g^m)^*: H^{1,1}(X,\R)\circlearrowleft$
(which is the essential part of the statement)
follows from Theorem~\ref{thm-filtintro}.
The proof of Theorem~\ref{t:main_1} is
based on a more refined dynamical filtration;
the reader is referred to \S\ref{s:length} for details.

See~\cite{LOZ} for other applications of dynamical filtrations.

\ssec{Nilpotency classes: a conjecture and some consequences}\label{ssec-nilcl}
\hfill

The {\it descending central series} of a group $H$ is defined by
$$\Gamma_0H := H \quad \text{and} \quad \Gamma_{i+1}H := [\Gamma_iH, H] = [H,\Gamma_i H]\, .$$
By definition, $\Gamma_pH = \{1\}$ for some integer $p \ge 0$ if and only if $H$ is nilpotent. We call the minimum of such $p$ the {\it nilpotency class} of $H$ (when $H$ is nilpotent) and denote it
by $c(H)$.
If $H$ is not nilpotent, we set $c(H) := \infty$.

Let $G \subset \Aut(X)$ be a zero entropy subgroup.
There exists a finite index subgroup $G' \subset G$ such that
the action of $G'/G'_0$ on $H^{p,p}(X,\R)$
	realizes $G'/G'_0$ as a unipotent subgroup of $\GL(H^{p,p}(X,\R))$ 
	for every $1 \le p \le \dim X -1$.
Similar to the definition of essential derived length,
we define the
 \emph{essential nilpotency class}
$$c_{{\rm ess}} (G, X) := c(G'/G'_0)$$
of $G \subset \Aut(X)$ to be the
nilpotency class of $G'/G'_0$;
according to Proposition~\ref{p:main_finind}, it does not depend on the choice of $G'$.

For any group $H$, the inclusion $H^{(i)} \subset \Gamma_iH$ implies $\ell(H) \le c(H)$.
We even have
$H^{(i)} = 0$ for the $i$-th derived group,
provided $2^i > c(H)$ (see \emph{e.g.} \cite[\S 5.1.12, Proof]{Ro}),
which implies
\begin{equation}\label{ineq-llogc}
\ell (H)\leq \lfloor \log_2(c(H)) \rfloor +1
\end{equation}
if $H$ is non-trivial.

We believe that the upper bound in Theorem \ref{t:main_1} is not optimal
and propose the following more precise conjecture involving the Kodaira dimension $\gk(X)$.
In view of~\eqref{ineq-llogc},
this conjecture would improve Theorem \ref{t:main_1} in a significant way:
the right hand side of the inequality there would be replaced by $\lfloor \log_2(n-1) \rfloor +1$.

\begin{conjecture} \label{conj:main}
	Let $X$ be a compact K\"ahler manifold of dimension $n \ge 1$.
	For every subgroup $G \subset \Aut(X)$ of zero entropy, we have
	$$c_\ess(G,X) \le n-  \max\{\gk(X), 1\}.$$
\end{conjecture}

We will construct some concrete examples in \S \ref{s:examples}
showing the optimality of the upper bounds in Theorem~\ref{t:main-1map} and
Conjecture \ref{conj:main}.

\begin{remark}
	In~\cite{EpsteinThurston}, Epstein and Thurston  studied upper bounds for the derived length of a solvable or nilpotent \emph{connected} Lie group acting continuously on a Hausdorff space. Their examples in \cite[Remark 1.5]{EpsteinThurston} may justify
 the necessity to quotient-away the continuous part $G_0$ of $G$ in
	Conjecture \ref{conj:main}. See also~\cite[\S 6.2]{CantatXie} for relevant results.
\end{remark}

\ssec*{Terminology and Notation}
\hfill

In this paper, we work in the category of analytic spaces.
All manifolds are assumed to be connected.
By Zariski closures, we also mean \emph{analytic} Zariski closures unless otherwise specified.
Given two sequences $(x_m)$ and $(y_m)$ in $\R_{\geq 0}$ or more generally in a salient convex closed cone,
the relation $x_m \lesssim y_m$ means $Cy_m-x_m$ belongs to this cone for some constant $C>0$ independent of $m$,
and $x_m \simeq y_m$ means
 $x_m \lesssim y_m$ and $y_m \lesssim x_m$.

When a group $N$ acts on a space $V$ preserving some structure on $V$,
we denote by $N|_V$ the image of $N$ in the group of automorphisms of $V$.
For instance, if $G$ is a group acting on a complex manifold $X$, then $G|_X$ is the image of $G$ in $\Aut(X)$, and $G|_{H^p(X, \Z)}$ is
the image of $G$ in $\GL(H^p(X, \Z))$ for the pullback action.
 For a normal subgroup $N_1 \unlhd N$,
we set
$(N/N_1)|_{V} = (N|_V)/(N_1|_V)$.

We say that a property holds for {\it very general} (resp. {\it general}) parameters or points if it holds for all parameters or points outside a countable
(resp. finite) union of proper closed analytic subvarieties of the space of parameters or points.

\ssec*{Acknowledgments}
\hfill

The authors are supported by NUS and MOE grants R-146-000-248-114 and MOE-T2EP20120-0010;
 Taiwan Ministry of Education Yushan Young Scholar Fellowship (NTU-110VV006),
 and Taiwan Ministry of Science and Technology (110-2628-M-002-006-);
 a JSPS Grant-in-Aid (A) 20H00111 and an NCTS Scholar Program; and an ARF of NUS, respectively. We would like to thank Fei Hu and Jun-Muk Hwang for helpful discussions; Serge Cantat for various comments and for bringing the references~\cite{CantatXie, EpsteinThurston} to our attention;
 the referees for very detailed and constructive suggestions; and KIAS, National Center for Theoretical Sciences (NCTS) in Taipei, NUS and the University of Tokyo for the hospitality and support during the preparation of this paper.

\section{Essential derived lengths and nilpotency classes} \label{s:unipot}

\ssec{Unipotent subgroups}
\hfill

Lemma \ref{l:unipotent} below should be well known, 
but we give a proof due to the lack of reference.

\begin{lemma} \label{l:unipotent}
Let $V$ be a real vector space of finite dimension. Let $\Gamma$ be a subgroup of $\GL(V)$. Assume that there is an integer $N\geq 1$ such that $g^N$ is unipotent for every $g\in\Gamma$.
Then there is a finite-index subgroup $\Gamma'$ of $\Gamma$ which is a unipotent subgroup of $\GL(V)$.
\end{lemma}

\begin{proof}
Let $\overline\Gamma$ denote the algebraic Zariski closure of $\Gamma$ in $\GL(V_\C)$,
where $V_\C:=V\otimes_\R\C$. This is a complex  linear algebraic group.
By hypothesis, for $g\in \Gamma$, all eigenvalues of $g$ are $N$-th roots of unity. It follows that the coefficients of the characteristic polynomial of $g$ belong to a finite set. These coefficients can be seen as polynomial functions in the matrix coefficients of $g\in \GL(V_\C)$. 
We deduce that the same property holds for all $g\in\overline\Gamma$.
Let $H$ be the component of the identity of $\overline\Gamma$. 
By continuity, the characteristic polynomial of
$g\in H$ is constantly equal to $(t-1)^{\dim V}$, as this is the case for the identity element. Thus $1$ is the only eigenvalue for all $g \in H$.

Define $\Gamma':=\Gamma\cap H$;  $\Gamma'$ is a unipotent subgroup of $\GL(V)$ by~\cite[\S17.5]{Hum}.
Since $\overline\Gamma$ is algebraic, it has a finite number of connected components. It follows that $H$ is a finite-index subgroup of $\overline\Gamma$ and $\Gamma'$ is a finite-index subgroup of $\Gamma$.
\end{proof}

\begin{thm} \label{l:index-ker}
Let $V_{\Z}$ be a free $\Z$-module of finite rank and let $\rho :\Gamma\to \GL(V_\Z)$ be a group homomorphism such that $\Ker (\rho)$ is finite and all elements of $\rho(\Gamma)$ are unipotent. Then
\begin{enumerate}
\item[(1)] (Mal'cev's theorem) $\Gamma$ is countable, finitely generated, and admits only countably many subgroups. Moreover, all subgroups of $\Gamma$ are finitely generated as well.
\item[(2)] There is a finite-index normal subgroup $\Gamma'$ of $\Gamma$ such that $\Gamma'\cap\Ker(\rho)=\{1\}$.
\end{enumerate}
\end{thm}

\begin{proof}
(1) By the assumption, $\rho(\Gamma)$ is a unipotent, hence solvable, subgroup of
$\GL(V_\Z)$ by~\cite[\S17.5]{Hum}. 
Thus the result follows from a theorem of Mal'cev~\cite[p.26, Corollary 1]{Polycicgp}.

\medskip

(2)
By Zorn's lemma, there exists a maximal normal subgroup  $\Gamma' \subset \Gamma$
such that $\Gamma'\cap\Ker(\rho)=\{1\}$. It is enough to show that $\Gamma'$ is of finite-index in $\Gamma$. Assume by contradiction that the group $K:=\Gamma/\Gamma'$ is infinite. Denote by $\pi:\Gamma\to K$ the canonical homomorphism. If $L$ is a normal subgroup of $K$, then $\pi^{-1}(L)$ is a normal subgroup of $\Gamma$ containing $\Gamma'$.
To complete the proof of the lemma, i.e. to get a contradiction, it is enough to construct such a group $L$ with $L\not=\{1\}$ and $L\cap \pi(\Ker(\rho))=\{1\}$.

Since $\rho(\Gamma)$ is a unipotent subgroup of
$\GL(V_\Z)$, we have a finite sequence
$$\Gamma =\Gamma_0 \vartriangleright \Gamma_1 \vartriangleright \cdots \vartriangleright\Gamma_{m+1} \subset \ker(\rho) \quad \text{with} \quad \Gamma_{j+1}:=[\Gamma,\Gamma_j]\,.$$
As remarked above, all such subgroups $\Gamma_j$ of $\Gamma$ are finitely generated.
Since $K$ is infinite and $\ker(\rho)$ is finite, there is a $j$ such that $\pi(\Gamma_j)$ is infinite but $\pi(\Gamma_{j+1})$ is finite.
So $\pi(\Gamma_j)/\pi(\Gamma_{j+1})$ is an infinite abelian group which is finitely generated. Thus, there is an element $a\in\pi(\Gamma_j)$  which is of infinite order.
Clearly,  $a^m$ does not belong to $\pi(\Ker(\rho))$ for $m\not=0$ because $\ker(\rho)$ is finite. We need the following:

\begin{claim}
For any $k \in K$, there is an integer $l > 0$ such that $ka^{ml}k^{-1}=a^{ml}$ for every $m$.
\end{claim}

We prove the claim.
By construction, $\pi(\Gamma_j)/\pi(\Gamma_{j+1})$ is contained in the center of $K/\pi(\Gamma_{j+1})$.
Therefore, for every $p\geq 1$, there is some $b_p = b_p(k) \in \pi(\Gamma_{j+1})$ such that $ka^pk^{-1}=b_pa^p$.
Since $\pi(\Gamma_{j+1})$ is finite, there are $p<q$ such that $b_p=b_q$. We then have for $l:=q-p$
$$ka^lk^{-1}=(ka^pk^{-1})^{-1}(ka^qk^{-1})=a^l\,.$$
This implies the claim for $m=1$ and then for every $m$.

\par \vskip 1pc
We resume the proof of
Theorem~\ref{l:index-ker}.
Since $K$ is finitely generated, the claim provides an $l\geq 1$ such that $ka^lk^{-1}=a^l$ for every $k\in K$. It follows that the group generated by $a^l$ is normal in $K$. Since it has a trivial intersection with $\pi(\Ker(\rho))$, we get a contradiction.
This ends the proof of Theorem~\ref{l:index-ker}.
\end{proof}

\ssec{Essential derived lengths and nilpotency classes}
\hfill

Lemma \ref{l:unipotent_property} is well known; we give a proof for completeness.

\begin{lemma}\label{l:unipotent_property}
Let $H$ be a unipotent subgroup of $\GL(N,\R)$ (or $\GL(N,\C)$), 
and let $H'$ be a finite-index subgroup of $H$. 
Then we have $\ell(H') = \ell(H)$ and $c(H') = c(H)$.
\end{lemma}

\begin{proof}
By~\cite[\S 17.5]{Hum},
$H$ can be regarded as a subgroup of the group of unipotent upper triangular matrices in $\GL(N,\C)$ for a suitable $N$.
Let $\overline{H} \subseteq \GL(N,\C)$ be the algebraic Zariski closure of $H$. Then $\overline{H}$ is still a unipotent subgroup of $\GL(N,\C)$ 
and hence connected~\cite[Exercise 15.5.6]{Hum}.

Let $H \supseteq H^{(1)} \supseteq H^{(2)} \supseteq \cdots$ be the derived series of $H$.
Then, ${\overline{H}}^{(i)}$ is equal to $\overline{H^{(i)}}$, the  algebraic Zariski closure of $H^{(i)}$,
see \emph{e.g.}~\cite[Proof of Lemma 2.1(2)]{Og06}.
It follows that the derived length of $H$ is equal to that of $\overline{H}$.

Since $H'$ is of finite-index in $H$, so is the group $\overline{H'}$ in the group $\overline{H}$.
Finally, since the latter two groups are both 
unipotent subgroups of $\GL(N,\C)$, they are connected, and hence equal.
Thus, we have $\ell(H) = \ell(\overline{H}) = \ell(\overline{H'}) = \ell(H')$.
By replacing the derived series by the central series, the same argument shows $c(H') = c(H)$.
\end{proof}

\begin{lemma} \label{l:Aut_unipotent}
Let $X$ be a compact K\"ahler manifold of dimension $n$. 
Let $G \subset \Aut(X)$ be a subgroup of zero entropy. Then there is a finite-index subgroup $G'$ of $G$ satisfying the following properties for every $1\leq p\leq n-1$.
\begin{enumerate}
\item The kernels of the canonical representations
$$\rho_p: G' \to \GL(H^{2p}(X,\R)) \qquad \text{and} \qquad \rho_{p,p} : G' \to \GL(H^{p,p}(X,\R)) $$
are both equal to $G_0'$.
\item The images of both $\rho_p$ and $\rho_{p,p}$ are unipotent subgroups of $\GL(H^{2p}(X,\R))$ and $\GL(H^{p,p}(X,\R))$, respectively.
\end{enumerate}
\end{lemma}

\begin{proof}
It is enough to prove the statement for a fixed $p$ because we can then deduce the lemma using a simple induction on $p$.

The argument of this paragraph is well-known.
Let $g$ be an automorphism of $X$ of zero entropy. For every $m\in\Z$, $g^m$ also has zero entropy. 
Therefore,  by Gromov--Yomdin's theorem, all eigenvalues of $\rho_p(g^m)$ are of modulus less than or equal to 1. Since $\rho_p(g^m)$ preserves the image of $H^{2p}(X,\Z)$ in $H^{2p}(X,\R)$, we deduce that the characteristic polynomial of $\rho_p(g^m)$ belongs to a finite family of polynomials, as their coefficients are bounded integers.
If $\lambda$ is an eigenvalue of $\rho_p(g)$ then $\lambda^m$ is an eigenvalue of $\rho_p(g^m)$ and hence belongs to a finite set which is independent of $m$. Thus, there is an integer $N\geq 1$ such that $\lambda^N=1$ for all eigenvalues of $\rho_p(g)$.

According to Lemma \ref{l:unipotent}, replacing $G$ by a finite-index subgroup, we have that $\rho_p(G)$ contains only unipotent elements of  $\GL(H^{2p}(X,\R))$. So we have the property (2) in the lemma for both $\rho_p$ and $\rho_{p,p}$.

The kernel $K$ of $\rho_{p,p}:G\to \GL(H^{p,p}(X,\R))$ is the set of $g\in G$ whose action on $H^{p,p}(X,\R)$ is trivial.
In particular, $K$ preserves the $p$-th power of a K\"ahler class.
By a generalized version of Fujiki--Lieberman's theorem \cite[Theorem 2.1]{DHZ},
the quotient $K/G_0$ is a finite group (for $p=1$, we can refer to the original theorem of Fujiki and Lieberman).
It follows that $\rho_p(K) \subset \GL(H^{2p}(X,\R))$ is finite and unipotent, thus trivial. 
This implies $K = \ker(\rho_p)$.

Finally, we apply
Theorem~\ref{l:index-ker} to $\Gamma:=G/G_0$ and $V_\Z = H^{2p}(X,\Z)/\torsion$.
According to this theorem, we can replace $G$ by a finite-index subgroup $G'$
and assume that $G'_0 = \ker({\rho_p}_{|G'}) = \ker({\rho_{p,p}}_{|G'})$.
This implies the property (1) for both $\rho_{p, p}$ and $\rho_{p}$.
\end{proof}

As immediate consequences of Lemmas \ref{l:unipotent_property} and \ref{l:Aut_unipotent},
we now have the following.

\begin{proposition}\label{p:main_finind}
	Let $X$ be a compact K\"ahler manifold of dimension $n \ge 1$
	and let $G \subset \Aut(X)$ be a zero entropy subgroup.
	We have the following assertions.
	\begin{itemize}
		\item[(1)]
		$G$ admits a finite-index subgroup $G'$ such that for any $1\leq p\leq n-1$,
		the natural map $G'/G'_0 \to G'|_{H^{2p}(X, \R)}$ (resp. $G'/G'_0 \to G'|_{H^{p,p}(X, \R)}$) embeds $G'/G'_0$ as a unipotent subgroup of $\GL(H^{2p}(X, \R))$ 
		(resp. $\GL(H^{p,p}(X, \R))$).
		\item[(2)] For every finite-index subgroup $G''$ of $G$ such that 
		$G''/G''_0 \to \GL(H^{p,p}(X, \R))$  is an isomorphism onto a
		unipotent subgroup for some $1\leq p\leq n-1$,
		the derived length $\ell (G''/G''_0)$ and the nilpotency class $c(G''/G''_0)$ of $G''/G''_0$ 
		do not depend on the choice of $G''$, nor on $p$.
	\end{itemize}
\end{proposition}

\begin{proof}
	(1) is contained in Lemma~\ref{l:Aut_unipotent}.
	(2) follows from Lemma~\ref{l:unipotent_property}, as it implies
	$\ell (G'/G'_0)  = \ell (G''/G''_0)$ and $c (G'/G'_0)  = c (G''/G''_0)$.
\end{proof}

\section{Dynamical filtrations and consequences} \label{s:length}

\ssec{Quasi-nef sequences}
\hfill

Let $\Kc^i(X)$ denote
the closure of the convex cone generated by the classes of smooth strictly positive closed $(i, i)$-forms in $H^{i,i}(X,\R)$. 
This is a salient  (that is, no line contained) convex closed cone with non-empty interior and is $\Aut(X)$-invariant.
In particular, $\Kc^1(X) = \Nef(X)$ is the nef cone.

For any class $L\in \Kc^i(X)\setminus\{0\}$ with $0\leq i\leq n-1$,
denote by $\Nef(L)$ the closure of the cone $L\cdot \Nef(X)$ in $H^{i+1,i+1}(X,\R)$.
\begin{remark}\label{rem-cex}
The image of a salient closed convex cone under a linear map is not necessarily closed,
see e.g.~\cite[Remark 2.5]{Loo}. 
Here is another example suggested by the referee.
The image of
$$C = \Set{(x,y,z) \in \R^3 | (y - z)^2 + x^2 \le z^2, y\geq 0, z\geq 0} \subset \R^3$$
	under the projection $\R^3 \to \R^2$ to the $xy$-plane is
	the open upper half-plane union $\{0\}$.
	For later use, observe that $C+\{x=y=0\}$ is not closed in $\R^3$.
\end{remark}

Define also for $0\leq i\leq n$
$$\Nc^i(X) := \Set{\gb \in H^{i,i}(X,\R) |  \gb \equiv 0 },$$
where we recall that $\gb \equiv 0$ means $\gb \cdot H^{1,1}(X,\R)^{n-i} = 0$.
This is a linear subspace of $H^{i,i}(X,\R)$ preserved by $\Aut(X)$.
Clearly $\Nc^n(X)=0$ and $\Nc^0(X)=0$.
We also have $\Nc^i(X) \cdot M \subset \Nc^{i+1}(X)$ for all $M \in H^{1,1}(X,\R)$.
We have $\Nc^{n-1}(X)=0$ by Poincar\'e duality,
and hence $\Nc^1(X)=0$ by hard Lefschetz's theorem 
(see e.g.~\cite[Theorem 6.25]{VoisinI}).

\begin{lemma} \label{l:Nef-N}
Let $L\in \Kc^i(X)\setminus\{0\}$ with $0\leq i\leq n-1$. Then:

\begin{itemize}
\item[{\rm (1)}] $\Nef(L)$ is a salient convex closed cone contained in the vector space $L\cdot H^{1,1}(X,\R)$ and in the cone $\Kc^{i+1}(X)$.

\item[{\rm (2)}] $\Kc^{i+1}(X)\cap \Nc^{i+1}(X)=\{0\}$ and $\Nef(L)\cap \Nc^{i+1}(X)=\{0\}$. In particular, $\Kc^{i+1}(X)+\Nc^{i+1}(X)$ and $\Nef(L)+\Nc^{i+1}(X)$ are closed in $H^{i+1,i+1}(X,\R)$.

\item[{\rm (3)}] Let $M_1,\ldots,M_p$ be in $H^{1,1}(X,\R)$ such that $LM_1,\ldots ,LM_p$ are in $\Nef(L)$. Then $LM_1\cdots M_p$ belongs to $\Kc^{i+p}(X)$. In particular, we have $LM_1\cdots M_p=0$ if and only if $LM_1\cdots M_pc_1\cdots c_{n-i-p}=0$ for some K\"ahler classes $c_1,\ldots, c_{n-i-p}$.

\item[{\rm (4)}] Let $M_1,\ldots,M_p$ be in $H^{1,1}(X,\R)$ such that $LM_1,\ldots ,LM_p$ are in $\Nef(L)+\Nc^{i+1}(X)$. Then $LM_1\cdots M_p$ belongs to $\Kc^{i+p}(X)+\Nc^{i+p}(X)$.
\end{itemize}
\end{lemma}

\proof
(1) By definition, $\Nef(L)$ is convex and closed.
Clearly, $L\cdot \Nef(X)$ is contained in $L\cdot H^{1,1}(X,\R)$ and in $\Kc^{i+1}(X)$ which are both closed.
It follows that $\Nef(L)$ is also contained in $L\cdot H^{1,1}(X,\R)$ and $\Kc^{i+1}(X)$.
Since $\Kc^{i+1}(X)$ is salient, $\Nef(L)$ satisfies the same property.

(2) Since each class $\ga\in \Kc^{i+1}(X)\setminus\{0\}$ can be represented by a non-zero positive closed current, we have
$\ga\cdot c_1\cdots c_{n-i-1}\not=0$ for all K\"ahler classes $c_1,\ldots, c_{n-i-1}$.
Thus $\Kc^{i+1}(X)\cap \Nc^{i+1}(X)=\{0\}$ and $\Nef(L)\cap \Nc^{i+1}(X)=\{0\}$.
The second assertion is then a direct consequence.

(3) 
For simplicity, we consider $p=2$ as the general case can be obtained in the same way.
Since $LM_1, LM_2$ are in $\Nef(L)$, there are nef classes $N_{1,k}, N_{2,\ell}$ such that $LN_{1,k}$ converges to $LM_1$ and
$LN_{2,\ell}$ converges to $LM_2$. We deduce that
$$LM_1M_2=\lim_{k\to\infty} L N_{1,k} M_2 = \lim_{k\to\infty} LM_2 N_{1,k}
= \lim_{k\to\infty}LN_{2,\ell_k} N_{1,k},$$
where we choose $l_k$ large enough so that $\|LN_{2,\ell_k} N_{1,k}-LM_2 N_{1,k}\|\leq 1/k$.
It is now clear that $LM_1M_2$ belongs to $\Kc^{i+2}(X)$.
The last assertion is true, as we just observed in (2),  for every class in $\Kc^{i+2}(X)$ and in particular for the class $LM_1M_2$.

(4) As above, we assume $p=2$ for simplicity. By hypothesis, there are nef classes $N_{1,k}, N_{2,\ell}$ and classes $N_1^{(0)}, N_2^{(0)} \in \Nc^{i+1}(X)$ such that
$$LM_1= \(\lim_{k\to\infty} LN_{1,k}\) + N_1^{(0)} \quad \text{and} \quad LM_2= \(\lim_{\ell \to\infty} LN_{2,\ell}\) + N_2^{(0)}.$$
For suitable $l_k$ going to infinity fast enough,
we have
$$LM_1M_2= \(\lim_{k\to\infty} L N_{1,k}M_2\) +N_1^{(0)}M_2 =  \lim_{k\to\infty}\( L N_{1,k} N_{2,l_k} + N^{(k)}\)$$
with $N^{(k)}:= N_{1,k} N_2^{(0)}+N_1^{(0)}M_2$. Since $L N_{1,k} N_{2,\ell} \in\Kc^{i+2}(X)$ and $N^{(k)}\in \Nc^{i+2}(X)$, using (2), we deduce from
the above identities that $LM_1M_2\in \Kc^{i+2}(X)+\Nc^{i+2}(X)$.
\endproof

Let $G \subset \Aut(X)$ be a subgroup.
Until the end of \S\ref{ssec-Polg}, we assume that
\textbf{the action $G \cto H^\bullet(X,\R)$ is unipotent}; see Lemma~\ref{l:Aut_unipotent}.

\begin{lemma} \label{l:quasi-nef}
There is a sequence of classes $L_i\in \Kc^i(X)\setminus \{0\}$ for $0\leq i\leq n$ such that:

\begin{enumerate}
\item $L_0=1$ and $L_{i+1}\in \Nef(L_i)$ for every $0\leq i \leq n-1$. In particular, there is a class $M_{i+1}\in H^{1,1}(X,\R)$ such that $L_{i+1}=L_iM_{i+1}$; and
\item $L_i$ is fixed by $G$ for every $0\leq i\leq n$.
\end{enumerate}
\end{lemma}

Note that in Lemma~\ref{l:quasi-nef}, for the reason similar to Remark~\ref{rem-cex},
	we cannot assume $M_{i+1}$ to be nef.

\proof
We construct the sequence by induction.
Assume that $L_0,\ldots, L_i$ are already constructed and satisfy the above properties (1) and (2). 
We see that $\Nef(L_i) \ne 0$ is invariant by $G$.
Since $G \cto H^\bullet(X,\R)$ is unipotent, it factors through a solvable group action.
Therefore, by a version of Lie--Kolchin's theorem for cone~\cite[Theorem 1.1]{KOZ}, there is a ray in $\Nef(L_i)$ which is preserved by $G$. Choose $L_{i+1}$ in this ray.
Since $1$ is the only eigenvalue of $g^* : H^{i+1,i+1}(X,\R) \cto$ for any $g\in G$, we deduce that $L_{i+1}$ is fixed by $G$.
This completes the proof of the lemma.
\endproof

The sequence $M_1,\ldots,M_n \in H^{1,1}(X,\R)$ as in Lemma~\ref{l:quasi-nef}
is called a \emph{quasi-nef} sequence.

\ssec{Dynamical filtrations} \label{ss:filtration}
\hfill

Let  $M_1,\ldots,M_n \in H^{1,1}(X,\R)$ be a quasi-nef sequence constructed
with respect to the $G$-action
(see Lemma~\ref{l:quasi-nef}).
Consider the following $G$-stable linear subspaces of $H^{1,1}(X,\R)$.
Set $F_n := H^{1,1}(X, \R)$ and define for $0\leq i\leq n-1$
\begin{equation}
\begin{split}
	F_i & \cnec \Set{M \in H^{1,1}(X,\R) | L_i M  \cdot H^{1,1}(X,\R)^{n-i-1} = 0} \\
	 & = \(L_i \cupp \bullet\)^{-1} \(\cN^{i+1}(X)\).
\end{split}
\end{equation}
They form an increasing filtration
$$\cdots \subset F_i \subset F_{i+1} \subset \cdots.$$
By Poincar\'e duality, in the definition of $F_i$ we can replace $n-i-1$ by
$n-i-2$ (when $i\leq n-2$) without changing this space.
We have $F_0= \cN^1(X) = 0$.
We also see that $F_{n-1}$ is the hyperplane
$$\Set{M \in H^{1,1}(X,\R) | L_{n-1} M  = 0} \subset H^{1,1}(X,\R).$$

Consider K\"ahler classes $c_1,\ldots,c_{n-i-2}$ and define the quadratic form $Q_i$ on $H^{1,1}(X,\R)$ by
\begin{equation}\label{def-Q}
Q_i(M,M'):= L_iMM'c_1\cdots c_{n-i-2} \quad \text{for} \quad M,M'\in H^{1,1}(X,\R)\,.
\end{equation}
Here, the right hand side is a real number since $H^{n,n}(X,\R)$ is canonically identified with $\R$. By the definition of $F_i$, this quadratic form induces a quadratic form on $H^{1,1}(X,\R)/F_i$ that we still denote by $Q_i$.
Consider another K\"ahler class  $c_{n-i-1}$ and the 
primitive subspace\footnote{This is a generalization of the classical primitive subspace
		$\Ker(c^{n-1} \cupp \bullet) \subset H^{1,1}(X,\R)$ for some K\"ahler class $c$. }
$$P_i:=\Set{M\in H^{1,1}(X,\R) |  L_iM c_1\cdots c_{n-i-1}=0 }.$$
Since $L_i$ belongs to $\Kc^i(X)\setminus\{0\}$, we have $L_i c_1\cdots c_{n-i-1}\not =0$. Therefore, by Poincar\'e duality, $P_i$ is a hyperplane of $H^{1,1}(X,\R)$.
It contains $F_i$, hence $P_i/F_i$ is a hyperplane of $H^{1,1}(X,\R)/F_i$.

\begin{lemma} \label{l:HR}
The quadratic form $Q_i$ is negative semi-definite on $P_i$,
hence also on $P_i/F_i$.
Moreover, $Q_i$ is negative semi-definite on $P_{i+1}/F_i$, and hence on $F_{i+1}/F_i$.
\end{lemma}

\proof
If we replace $L_i$ by a product of K\"ahler classes, then by the mixed version of Hodge--Riemann theorem, $Q_i$ is negative definite on $P_i$, see~\cite[Theorem A]{DN} and~\cite{Gromov1}. By definition, $L_i$ is a limit of classes which are products of K\"ahler classes. Thus, by continuity, we obtain the first assertion.

We have seen that
$$P_{i+1} = \Set{M\in H^{1,1}(X,\R) | L_{i+1}M c_1\cdots c_{n-i-2}=0 }.$$
is a hyperplane of $H^{1,1}(X,\R)$.
Recall that $L_{i+1}$ can be approximated by classes $\{L_ic'_k\}_{k \ge 1}$  with $c'_k$ K\"ahler.
Therefore, we can approximate $P_{i+1}$ by
$$P'_k \cnec \Set{M \in H^{1,1}(X,\R) | L_i M c'_k c_1\cdots c_{n-i-2}=0}.$$
Finally, since $Q_i$ is negative semi-definite on $P'_k$,
by continuity, it is negative semi-definite on $P_{i+1}$ as well. 
This implies the second assertion of the lemma because $F_{i+1}$ is contained in $P_{i+1}$.
\endproof

\begin{lemma} \label{l:primitive}
	Let $M$ and $M'$ be classes in $H^{1,1}(X,\R)$
	such that
	$$Q_i(M,M')=Q_i(M',M')=0.$$
	Assume also that  $ L_iM \in \Nef(L_i)\setminus\{0\}$.
	Then there is a vector $(a,b)\in \R^2\setminus \{0\}$, unique up to a multiplicative constant, such that
	$$L_i(aM+bM')c_1\cdots c_{n-i-2}=0\,.$$
	If the subspace spanned by $M$ and $M'$ is a plane, then $\Vect(M,M')\cap P$ is the real line spanned by $aM+bM'$.
\end{lemma}

\proof
Since the classes $c_k$ are K\"ahler, the hypothesis on $ L_iM$ implies that
$$L_iMc_1\cdots c_{n-i-2}c_{n-i-1} \not=0 \quad \text{and} \quad L_iMc_1\cdots c_{n-i-2} \not=0\,.$$
Therefore, we have $M\not\in P$ and $(a,b)$ is unique up to a multiplicative constant.
The lemma is clear when $M$ and $M'$ are collinear.
Assume now that $M$ and $M'$ are not collinear.
Since $\Pi \cnec \Vect(M,M')$ is a plane,
$\Pi\cap P$ is a line. Let $N:=aM+bM'$ be a vector which spans this real line.
It is enough to show that this vector satisfies the identity in the lemma.

By Poincar\'e duality, we only have to show that $Q_i(N,N')=0$ for every $N'\in H^{1,1}(X,\R)$.
Since $N\in \Pi\cap P$, Lemma~\ref{l:HR} implies $Q_i(N,N)\leq 0$.
We also have
$$Q_i(N,N)=a^2Q_i(M,M)\geq 0,$$
because
$L_iM^2 \in \Kc^{i+2}(X)$ by Lemma~\ref{l:Nef-N}(3) and the hypothesis $ L_iM \in \Nef(L_i)$. We deduce that
$Q_i(N,N)=0$, and that $Q_i(M,M)=0$ when $a\not =0$.
In any case, we have $Q_i(N,N')=0$ for $N'\in\Pi$ because $Q_i(M,M')=0$.

Note that $H^{1,1}(X, \R) = \Pi + P$. Now, it suffices to check the identity $Q_i(N,N')=0$ for $N'\in P$.
Since $Q_i(N,N)=0$, using Lemma \ref{l:HR} and Cauchy-Schwarz's inequality, we have for $N'\in P$
$$Q_i(N,N')^2\leq Q_i(N,N)\cdot Q_i(N',N') =0\,.$$
The lemma follows.
\endproof

For $1\leq i\leq n$, define $F_i'$ as the vector subspace of $F_i\subset H^{1,1}(X,\R)$ spanned by the cone
$$C_{i} := \Set{ M \in F_i | L_{i-1}M \in \Nef(L_{i-1})+\Nc^i(X) }.$$
We can deduce from Lemma~\ref{l:codim-1} below that either $C_i=F_{i-1}$ or $C_i$ is a closed half-space of $F_i'$ having $F_{i-1}$ as the boundary.
We have
$$F_{i-1}\subset C_i\subset F_i' \subset F_i$$
since $L_{i-1}F_{i-1}\subset \Nc^i(X)$. Both $C_i$ and $F_i'$ are $G$-invariant. 
In particular, we have a $G$-invariant filtration
 \begin{equation}\label{eqn-filtFF'}
 \cdots \subset F'_{i-1} \subset F_{i-1} \subset F_i' \subset F_{i} \subset \cdots.
 \end{equation}
We also have $F_n'=H^{1,1}(X,\R)$ since $F_n=H^{1,1}(X,\R)$; thus $C_n$ contains the nef cone $\Nef(X)$ of $X$.

\begin{lemma} \label{l:Ci-Ci}
For $1\leq i\leq n$, we have $C_i\cap (-C_i)=F_{i-1}$.
\end{lemma}

\proof
Clearly, $F_{i-1}$ is contained in $C_i\cap (-C_i)$.
For $M\in C_i$, there are nef classes $N_k$ and a class $N^{(0)}\in \Nc^i(X)$ such that
\begin{equation*}
L_{i-1}M = \lim_{k\to\infty} L_{i-1}N_k+N^{(0)}.
\end{equation*}
In particular, if $c_1,\ldots,c_{n-i}$ are K\"ahler classes, then $L_{i-1}Mc_1\cdots c_{n-i}\geq 0$.
Assume now that $M\in C_i\cap (-C_i)$.
We obtain that $L_{i-1}Mc_1\cdots c_{n-i}= 0$ for every K\"ahler classes $c_1,\ldots, c_{n-i}$.
Thus, $L_{i-1}M H^{1,1}(X,\R)^{n-i}=0$ and $M\in F_{i-1}$.
\endproof

Recall that $L_i = L_{i-1}M_i$ by the definition of $M_i$. Note also that $\dim F_n'/F_{n-1}=1$ because $F_{n-1}$ is a hyperplane of
$F_n'=H^{1,1}(X,\R)$.
Since $M_n\not\in F_{n-1}$, we see that $F_n' = F_{n-1} \oplus \R \cdot M_n$. 

\begin{lemma} \label{l:codim-1}
Let $1\leq i\leq n$. We have the following properties.
\begin{enumerate}
\item $\dim F_{i}'/F_{i-1}\leq 1$ and $F_i' = \Set{ \ga \in F_i |L_{i-1}\ga^2 \equiv 0}$.
\item The following assertions are equivalent:
\begin{enumerate}
	\item $\dim F'_i/F_{i-1} = 1$;
	\item $F'_i = F_{i-1} \oplus \R \cdot M_i$;
	\item $L_{i-1} M^2_i \equiv 0$; 
	\item $L_{i-1} M^2_i = 0$;
	\item $C_i$ is the closed half-space of $F_i'$ containing $M_i$ and having $F_{i-1}$ as the boundary.
\end{enumerate}
\item For every $\ga \in F_i'$ (resp. $\ga \in C_i\setminus F_{i-1}$), $L_{i-1}\ga$ is proportional (resp. positively proportional) to $L_i$ modulo $\Nc^i(X)$.
\end{enumerate}
\end{lemma}

\begin{proof}

Lemma~\ref{l:codim-1} for $i = n$ follows from 
	$$F_n' = F_{n-1} \oplus \R \cdot M_n = H^{1,1}(X,\R).$$
From now on, we assume that $i \le n-1$.

(1)  Let $Q_{i-1}$ be the quadratic form defined using $L_{i-1}c_1\cdots c_{n-i-1}$ for some K\"ahler classes $c_j$.
By the second assertion of Lemma \ref{l:HR},  if $M'\in F_i$ then $Q_{i-1}(M',M')\leq 0$.
If $M'\in C_{i}$,
we have $Q_{i-1}(M',M')\geq 0$ according to Lemma \ref{l:Nef-N}(4).
Therefore,  we have $Q_{i-1}(M',M')=0$ for $M'\in C_i$ and hence for $M'\in F_i'$ as well
because  $C_i$ spans $F_i'$.
This proves
$$F_i' \subset \Set{ \ga \in F_i |L_{i-1}\ga^2 \equiv 0}.$$
The other inclusion and $\dim F_{i}'/F_{i-1}\leq 1$
follow easily from the following claim.

\medskip\noindent
{\bf Claim.} Let $M'$ be any class in $F_i$ satisfying $Q_{i-1}(M',M')=0$ for all K\"ahler classes $c_1,\ldots, c_{n-i-1}$. Then $M'$ is proportional to $M_i$ modulo $F_{i-1}$.

\begin{proof}[Proof of Claim]
	
Assume to the contrary that $M'$ is not proportional to $M_i$ modulo $F_{i-1}$.
By the definitions of $F_i$ and $M_i$, we have $Q_{i-1}(M_i,M')=0$; we also have $Q_{i-1}(M',M')=0$ by hypothesis.
According to Lemma \ref{l:primitive},
there is a vector $(a,b)\in \R^2\setminus \{0\}$, unique up to a multiplicative constant, such that
\begin{equation} \tag{A}
L_{i-1}(aM_i+bM')c_1\cdots c_{n-i-1}=0.
\end{equation}
Also since $M'$ and $M_i$ span a plane, by the second assertion of Lemma \ref{l:primitive}, this  $(a,b)$ is  the unique (up to a multiplicative constant) solution of the equation
\begin{equation}
\tag{B}
L_{i-1}(aM_i+bM')c_1\cdots  c_{n-i-1}c_{n-i}=0.
\end{equation}
So Equation (B) is equivalent to Equation (A) which can be obtained from (B) by removing the factor
$c_{n-i}$. Since this property holds for arbitrary K\"ahler classes $c_1,\ldots,c_{n-i}$, Equation (B) is equivalent to any equation obtained from (B) by removing a factor $c_j$:
\begin{equation} \tag{A'}
L_{i-1}(aM_i+bM')c_1\cdots c_{j-1}c_{j+1} \ldots c_{n-i-1}c_{n-i}=0.
\end{equation}

We conclude that (A) and (A') are equivalent. In other words, Equation (A) remains equivalent if we replace a factor $c_j$ by any other K\"ahler class. Therefore, if $(a,b)$ is as above, after replacing one by one the factors $c_j$ with arbitrary K\"ahler classes, we have
$$L_{i-1}(aM_i+bM')c_1\cdots c_{n-i-1}=0$$
for all K\"ahler classes  $c_1,\ldots , c_{n-i-1}$.
Since K\"ahler classes span $H^{1,1}(X,\R)$, we obtain
$$L_{i-1}(aM_i+bM') H^{1,1}(X,\R)^{n-i-1}=0\,.$$
Hence $M'$ has to be proportional to $M_i$ modulo $F_{i-1}$.	
\end{proof}
		
(2) The above claim implies that (a) implies (b).
	That (b) implies (c) follows from
	Part (1) of the lemma.
	As $L_{i-1} M_i = L_i \in \Nef(L_{i-1}) \subset \cK^i(X)$,
	it follows from Lemma~\ref{l:Nef-N}(3) that (c) is equivalent to (d).
	We show that (c) implies (a).
	Since $\dim F_{i}'/F_{i-1}\leq 1$, it suffices to show that $F_{i-1} \ne F_i'$.
	Assume that $L_{i-1}M_i^2 \equiv 0$, then $M_i \in F_i$
	because $L_iM_i = L_{i-1}M_i^2 \equiv 0$, and therefore $M_i \in F_i'$
	by Part (1) of the lemma.
	As $L_{i-1}M_i = L_i \not\equiv 0$, we have $M_i \not\in F_{i-1}$. Hence $F_{i-1} \ne F_i'$
	and (c) implies (a).
	It is clear that (e) implies (a). Finally, assuming (a), (b), (c), it is clear that $M_i\in C_i$, $M_i\not\in F_{i-1}$ and hence (e) is true, thanks to Lemma \ref{l:Ci-Ci}.

\smallskip

(3) Recall that $L_i=L_{i-1}M_i$. The assertion is clear when $F_i'=C_i=F_{i-1}$. Otherwise, this is a consequence of (b) and (d) in Part (2).
\end{proof}

\ssec{Polynomial growth of the pullback action: proofs of Theorems~\ref{t:main-1map} and~\ref{thm-filtintro}}\label{ssec-Polg}
\hfill

Let $M_1,\ldots,M_n \in H^{1,1}(X,\R)$ be a quasi-nef sequence with respect to the $G$-action.
Let $F_i$ and $F'_i$ be the filtrations constructed previously.

Let $g \in G$ and let $\go$ be a K\"ahler class.
We have the expansion:
\begin{equation}\label{e:gm-omega}
 (g^m)^*(\go)= \big((g^*-\Id)+\Id\big)^m(\go)= \sum_{0\leq j\leq m} {m\choose j} \go_j \quad \text{with} \quad \go_j:=(g^*-\Id)^j(\go).
\end{equation}
The following proposition combined with Lemma~\ref{l:codim-1} proves
Theorem~\ref{thm-filtintro}.

\begin{proposition} \label{p:decomposition}
	Recall that $g^* : H^{1,1}(X) \cto$ is unipotent.
	Assume that  $g^* \not=\Id$. Then there exist an integer  $1\leq r\leq n-1$ and
	a unique strictly decreasing sequence of integers
	$$n-1\geq s_1>\cdots > s_r \ge  1$$
	depending on $g$ but independent of the K\"ahler class $\go$, such that $\go_{2j-1}\in F_{s_j}\setminus F_{s_j}'$ and $\go_{2j}\in F_{s_j}'\setminus F_{s_j-1}$ for $1\leq j\leq r$, and $\go_{2r+1}=0$. Moreover, we have that $\dim F_{s_j}'/F_{s_j-1}=1$, $\go_{2j}$ is positively proportional to $M_{s_j}$ modulo $F_{s_j-1}$, and
	$F_{s_j}' = F_{s_j-1} \oplus \R \cdot M_{s_j}$.
\end{proposition}

For later use (after Lemma~\ref{l:Sobolev}), in order to recall that the $s_j$
in Proposition~\ref{p:decomposition} depends on $g \in \Aut(X)$, we will sometimes use the notation $s_j(g):=s_j$.
We will also use the notation $\go_1(g) \cnec \go_1$ in the expansion~\eqref{e:gm-omega} from time to time.

Before proving Proposition~\ref{p:decomposition}, we first deduce the upper bounds stated in Theorem \ref{t:main-1map}.

\begin{proof} [Proof of Theorem \ref{t:main-1map} (for the optimality, see \S \ref{ssec-exopt})]
Observe that it is enough to prove this theorem for any positive iterate of $g$ instead of $g$.
Since $g$ has zero entropy,
up to replacing $g$ by some positive iterate,
we can assume by Lemma~\ref{l:Aut_unipotent} that $g^* : H^{1,1}(X) \cto$ is unipotent.
We apply the above study for $G \subset \Aut(X)$ the cyclic group generated by $g$.
We deduce from Proposition~\ref{p:decomposition} that
for $m\geq 2r$,
$$(g^m)^*(\go)=\sum_{0\leq j\leq 2r} {m\choose j} \go_j.$$
Since $r\leq n-1$, we obtain that $\|(g^m)^*(\go)\|\lesssim m^{2(n-1)}$ and hence
$\|(g^m)^*(\go^p)\|\lesssim m^{2p(n-1)}$.
Since $\go^p$ is in the interior of the cone $\Kc^p(X)$, it follows that
(see also \cite[Proposition  5.8]{Dinh-bis})
$$\|(g^m)^*:H^{p,p}(X,\R)\circlearrowleft\| \lesssim  m^{2p(n-1)}.$$

Since $(g^m)^*:H^{p,p}(X,\R)\circlearrowleft$ is dual to $(g^{-m})^*:H^{n-p,n-p}(X,\R)\circlearrowleft$, we deduce that
$$\|(g^m)^*:H^{p,p}(X,\R)\circlearrowleft\| \lesssim  m^{2(n-p)(n-1)}.$$
Thus,
$$\|(g^m)^*:H^{p,p}(X,\R) \circlearrowleft \| \lesssim  m^{2p'(n-1)}$$
with $p' = \min(p,n-p)$.
Finally, we obtain the estimate given in Theorem \ref{t:main-1map} and conclude its proof by using:
$$\|(g^m)^*:H^{p,q}(X,\R)\circlearrowleft\|^2 \leq C \|(g^m)^*:H^{p,p}(X,\R)\circlearrowleft\|\|(g^m)^*:H^{q,q}(X,\R)\circlearrowleft\|$$
for some constant $C>0$ independent of $g$ and $m$, see the proof of \cite[Proposition  5.8]{Dinh-bis}.
\end{proof}

\begin{remark}\label{r:growth_Op}
	\hfill
	
	\noindent (1)
	A well-known result says that if $\|(g^m)^*:H^{1,1}(X,\R) \circlearrowleft\| \sim m^{k}$ when $m \gg 1$, then $k =2r$  is always an \emph{even} integer. Indeed, if $k$ is odd, then for a general K\"ahler class $\go$, we see from the expansion \eqref{e:gm-omega} that the K\"ahler classes $(g^{m})^*(\go)$ and $(g^{-m})^*(\go)$ would be asymptotically opposite. This contradicts the property that the cone $\Nef(X)$ is salient.
	
	\noindent (2) For the other values of $(p,q)$, the estimate can be slightly improved. Indeed, this is clear when $r<n-1$. When $r=n-1$, we have $s_i=n-i$ for every $i$ and  we can take $L_1=M_1=\go_{2r}$. Then, by Lemma \ref{l:codim-1} and Proposition \ref{p:decomposition}, we get
	$\go_{2r}^2\equiv 0$ and $\go_{2r}\go_{2r-1}\equiv 0$.
	Thus the following holds
	whenever $p\geq 2$ and regardless of the value of $r$:
	$$\|(g^m)^*:H^{p,p}(X,\R)\circlearrowleft\| =O(m^{2p(n-1)-2}).$$
\end{remark}

\begin{proof}[Proof of Proposition~\ref{p:decomposition}]

First observe that $g^*=\Id$ on $F_i'/F_{i-1}$ because $g^*$ is unipotent and $\dim F_i'/F_{i-1}\leq 1$ according to Lemma \ref{l:codim-1}. In particular, if $\go_k\in F_i'\setminus F_{i-1}$ then $\go_{p}\not\in F_i'$ for $p<k$ and $\go_{k+1}\in F_{i-1}$.
This remark will be used several times in the proof.

Starting from a  K\"ahler class $\go$,
we first construct by induction the sequence $s_j$ which verifies
the properties in Proposition~\ref{p:decomposition} except its independence of $\go$.
We will show later that $r$ and $s_j$ do not depend on the choice of $\go$.

We have $\go_0=\go\in F_n'=H^{1,1}(X,\R)$ and for simplicity we set $s_0:=n$.
Suppose that $j \ge 1$ and the sequence $(s_k)$ is already constructed for $k$ up to $j-1$, 
so that
$\go_{2i-1}\in F_{s_i}\setminus F_{s_i}'$ and $\go_{2i}\in F_{s_i}'\setminus F_{s_i-1}$ for $i=1,\ldots, j-1$.
For $i = j-1$ (when $j > 1$), this gives
$$\go_{2j-3} \in F_{s_{j-1}}\setminus F_{s_{j-1}}' \ \ \ \text{ and }
 \ \ \ \go_{2j-2}\in F_{s_{j-1}}'\setminus F_{s_{j-1}-1}.$$
If  $\go_{2j-1}=0$ (in particular, when $s_{j-1}=1$) we end the construction and define $r:=j-1$.
Assume now that $\go_{2j-1}\not=0$, so $s_{j-1}>1$.
It is enough to show the existence of $1\leq s_j<s_{j-1}$ such that $\go_{2j-1}\in F_{s_j}\setminus F_{s_j}'$ and $\go_{2j}\in F_{s_j}'\setminus F_{s_j-1}$.

Let $s_j\geq 1$ be the integer such that $\go_{2j-1}\in F_{s_j}\setminus F_{s_j-1}$.
By Lemma \ref{l:codim-1}, the space $F_{s_{j-1}}'/F_{s_{j-1}-1}$ is of dimension 0 or 1. Therefore,
since $\go_{2j-1}=(g^*-\Id)(\go_{2j-2})$ and
$\go_{2j-2}$ belongs to $F_{s_{j-1}}'\setminus F_{s_{j-1}-1}$, we necessarily have $\go_{2j-1} \in F_{s_{j-1}-1}$.
It follows that $s_j < s_{j-1}$.

Let $k\geq 2j-1$ be the largest integer such that $\go_k$ belongs to $F_{s_j}\setminus F_{s_j-1}$. So we have $L_{s_j-1}\go_k \cdot H^{1,1}(X,\R)^{n-s_j}\not=0$.
We only need to show that $\go_k\in F'_{s_j}$ and $k=2j$.
Note that these properties imply that $\go_{2j-1}\not\in F_{s_j}'$ thanks to the remark at the beginning of the proof.

Observe that modulo $F_{s_j-1}$ we have
$$(g^m)^*(\go) = \go_0 +{m\choose 1} \go_1 +\cdots + {m\choose k}\go_k.$$
Therefore, there are classes $N^{(m)}$ in  $L_{s_j -1}F_{s_j-1}\subset \Nc^{s_j}(X)$ such that
\begin{equation*} \label{e:go-k}
L_{s_j-1} \go_k=\lim_{m\to\infty} \((k!) m^{-k} L_{s_j-1} (g^m)^*(\go) + N^{(m)}\).
\end{equation*}
 Lemma \ref{l:Nef-N}(2) implies that $L_{s_j-1} \go_k$ is in $\Nef(L_{s_j-1})+\Nc^{s_j}(X)$. Since $\go_k\in F_{s_j}$, we conclude that
$\go_k$ belongs to $C_{s_j}\subset F'_{s_j}$.
As $\go_k\not\in F_{s_j-1}$,
by Lemma \ref{l:codim-1}(3), $\go_k$ is positively proportional to $M_{s_j}$ in $F'_{s_j}/ F_{s_j-1}$.

Since $\go_k \in F'_{s_j} \bss F_{s_j-1}$, 
Lemma \ref{l:Ci-Ci} implies that $-\go_k$ doesn't belong to $C_{s_j}$.
Since $(g^{-m})^*= \(\Id - (g^*)^{-1} \( g^* - \Id\) \)^m$,
we have
$$(g^{-m})^*(\go) = \go_0 -{m\choose 1} (g^{-1})^*(\go_1) +\cdots + (-1)^k {m\choose k}(g^{-k})^*(\go_k) \in H^{1,1}(X,\R) / F_{s_j-1}$$
which implies
 $$(g^{k-m})^*(\go) = {(g^k)^*(\go_0)} -{m\choose 1} (g^{k-1})^*(\go_1) +\cdots + (-1)^k {m\choose k}\go_k \in H^{1,1}(X,\R) / F_{s_j-1}.$$
Hence for some classes $N^{[m]}$ in  $L_{s_j -1}F_{s_j-1}\subset \Nc^{s_j}(X)$,
\begin{equation*}\label{eqn-limmk}
(-1)^kL_{s_j-1} \go_k=\lim_{m\to\infty} \((k!) m^{-k} L_{s_j-1} (g^{k-m})^*(\go)+N^{[m]}\).
\end{equation*}
It follows as above that $(-1)^k\go_k$ belongs to $C_{s_j}$. Thus, $k$ is even and hence $k\geq 2j$.

To show that $k=2j$, we assume by contradiction that $k\geq 2j+2$. Observe that, for $q\geq 2j-1$ 
we have $\go_q\in F_{s_j}$ by the definition of $s_j$.  
Therefore $L_{s_j-1}M_{s_j} \go_q=L_{s_j}\go_q \in \Nc^{s_j+1}(X)$, 
and hence $L_{s_j-1}\go_k \go_q  \in \Nc^{s_j+1}(X)$ 
(where we recall that $L_{s_j-1}\go_k$ is positively proportional to $L_{s_j-1}M_{s_j}$ modulo $\Nc^{s_j}(X)$). 
In particular, this holds for $q\geq k-2$. We also observe that by the definition of $k$, we have $\go_q \in F_{s_j - 1}$ whenever $q \geq k+1$, so $L_{s_j - 1}\go_q  \in \Nc^{s_j}(X)$.
Thus, we have (using the expansion \eqref{e:gm-omega})
$$\lim_{m\to\infty} \((k-1)!^2m^{-2k+2} L_{s_j-1} (g^m)^*(\go^2)\) = L_{s_j-1}\go_{k-1}^2 \ \  \text{ modulo }   \Nc^{s_j+1}(X).$$
When we multiply by K\"ahler classes,
the left hand side gives a non-negative number while the right hand side
gives a non-positive number due to Lemma \ref{l:HR} (because $\go_{k-1} \in F_{s_j}$).
We then obtain that $L_{s_j-1}\go_{k-1}^2=0$ modulo $\Nc^{s_j+1}(X)$. 
This contradicts Lemma~\ref{l:codim-1}(1) and the fact that
$\go_{k-1}\in F_{s_j}\setminus F_{s_j}'$, 
the latter due to $\go_k \in F'_{s_j} \bss F_{s_j-1}$
	and the remark at the beginning of the proof. 
Thus $k=2j$ as desired.
Finally, the fact that $\omega_k$ is positively proportional to $M_{s_j}$ in $F'_{s_j}/F_{s_j-1}$
together with Lemma~\ref{l:codim-1}(2)
shows that $s_j$ satisfies the "moreover part" of Proposition~\ref{p:decomposition}.

We prove now that the sequence $s_j$ doesn't depend on the choice of $\go$.
Let $\go'$ be another K\"ahler class. 
We have
$\go\lesssim \go'\lesssim \go$
which implies that
$$ L_i (g^m)^*(\go) \lesssim L_i (g^m)^*(\go') \lesssim L_i (g^m)^*(\go)\, .$$
It follows from the polynomial growth (when $m \to \infty$) of the expansion
$$L_{i}c_1\cdots c_{n-i - 1}(g^m)^*(\go)=\sum_{0\leq j\leq 2r} {m\choose j} L_{i}c_1\cdots c_{n-i - 1} \go_j
\ \ \ \ \ \ (c_1, \ldots, c_{n-i - 1} \text{ K\"ahler}) $$
and the similar one for $\go'$
that $k$ is the smallest integer such that $L_i\go_k \in \cN^{i + 1}(X)$
if and only if $k$ is the smallest integer such that $L_i\go'_k \in \cN^{i + 1}(X)$.
So for any integer $k$,  $\go_k \in F_i$ if and only if $\go'_k \in F_i$,
which shows that $s_j$ is independent of $\go$.

Finally, the uniqueness of the sequence $(s_k)$
is clear from the property $\go_{2j-1}\in F_{s_j}\setminus F_{s_j}'$
and~\eqref{eqn-filtFF'}.
This completes the proof of Proposition \ref{p:decomposition}
and also that of Theorem~\ref{thm-filtintro}.
\end{proof}

\ssec{Essential derived length: proof of Theorem~\ref{t:main_1}}  
\hfill

Thanks to Lemmas \ref{l:unipotent_property} and \ref{l:Aut_unipotent}, in order to prove
Theorem~\ref{t:main_1}, we may and will replace $G$ by a suitable finite-index subgroup and assume that the properties (1) and (2) in  Lemma \ref{l:Aut_unipotent} hold for $G$ instead of $G'$.

From now on, we fix some classes $L_i$ with $0 \le i \le n$ and
$M_i$ with $1 \le i \le n$ as in Lemma \ref{l:quasi-nef}.
Define the group $H_i$ by
\begin{eqnarray*}
	H_i &:=& \big\{g\in G \quad : \quad g^*=\Id\,\, {\rm on}\,\, L_i\cdot H^{1,1}(X,\R)/\Nc^{i+1}(X)\big\} \\
	&=& \big\{g\in G \quad : \quad  L_i\cdot (g^*-\Id) H^{1,1}(X,\R) \subset\Nc^{i+1}(X)\big\},
\end{eqnarray*}
where the quotient in the first display is understood as the quotient of $L_i\cdot H^{1,1}(X,\R)$ by its intersection with $\Nc^{i+1}(X)$.
Equivalently,
\begin{eqnarray} \label{e:Hi-Fi}
H_i &=& \big\{g\in G \quad : \quad (g^*-\Id)H^{1,1}(X,\R)\subset F_i \big\} \\
&=& \big\{g\in G \quad : \quad g^*=\Id \ \ \text{on} \ \ H^{1,1}(X,\R)/F_i \big\}. \nonumber
\end{eqnarray}
Since $F_{i+1} \supset F_i$,  the sequence $H_i$ is increasing,
$H_i$ is normal in $G$ and hence in $H_j$ for $i \le j$. 
Moreover, $G/H_i$ acts faithfully on  $L_i\cdot H^{1,1}(X,\R)/\Nc^{i+1}(X)$.
Also $H_0=G_0$, because Lemma \ref{l:Aut_unipotent} holds for $G$ instead of $G'$. So
if $H_0=G$ or equivalently if $G$ acts trivially on $H^{1,1}(X,\R)$,
then Theorem~\ref{t:main_1} holds obviously.

\begin{lemma} \label{l:max-length}
	Assume that $H_0\not=G$. Then there is an integer $0\leq l\leq n-2$ such that  $H_l\not= G$ and  $H_{l+1}=G$.
\end{lemma}

\proof
Since $G$ acts trivially on $H^{n,n}(X,\R)$,
$H_{n-1}=G$. The lemma follows easily.
\endproof

For $1\leq j\leq n-1$, we define
$$W_j \cnec \Set{\ga \in F_j |
	\begin{array}{l}
	\text{there is a class } \gb \in H^{1,1}(X,\R) \text{ such that } \\
L_{j-1} (M_j\gb +\ga^2)c_1\cdots c_{n-j-1}\geq 0
\text{ for all K\"ahler classes } c_i
\end{array}}.$$

Note that we obtain the same space if we assume that $\gb$ is K\"ahler, 
because when the inequality in the definition of $W_j$ holds for $\gb$ it holds for all $\gb'$ such that $\gb'-\gb$ is nef.
We call $W_j$ a \emph{primitive root space}\footnote{$\ga$ can be regarded as a square root of a class which is bounded below by $-M_j\gb$ in some sense.}.

\begin{lem}
For $1\leq j\leq n-1$, $W_j$ is a $G$-invariant vector space containing $F_j'$.
\end{lem}

\begin{proof}
	 It is clear that $W_j$ is invariant by $G$ and $W_j \supset F_j'$, see Lemma \ref{l:codim-1}.
	 Moreover, $W_j$ is stable under scalar multiplications, so we only need to show that $W_j + W_j \subset W_j$. Fix K\"ahler classes $c_1,\ldots, c_{n-j-1}$ and let $Q_{j-1}$ be the quadratic form~\eqref{def-Q} defined using $L_{j-1}c_1\cdots c_{n-j-1}$. By Lemma~\ref{l:HR}, $Q_{j-1}$ is negative semi-definite on $F_{j}$, so the function $F_{j} \ni v \mapsto Q_{j-1}(v,v)$ is concave.
	 In particular, for every $\ga_1,\ga_2 \in W_j$,
\begin{eqnarray*}
Q_{j-1}(\ga_1+\ga_2,\ga_1+\ga_2) &\ge & \frac{1}{2}\(Q_{j-1}(2\ga_1,2\ga_1) +Q_{j-1}(2\ga_2,2\ga_2)\) \\
&= & 2(Q_{j-1}(\ga_1,\ga_1) +Q_{j-1}(\ga_2,\ga_2)).
\end{eqnarray*}
It follows that $\ga_1+\ga_2\in W_j$.
\end{proof}

\begin{lemma} \label{l:Sobolev}
Let $\go$ be a K\"ahler class. Then, the class
$$\go_{2j-1} = (g^* - \Id)^{2j-1}(\go)$$
belongs to $W_{s_j}$ for every $1\leq j\leq r$.
\end{lemma}

\proof
We have for every K\"ahler classes $c_i$
$$m^{-4j+2}L_{s_j-1} (g^m)^*(\go^2)c_1\cdots c_{n-s_j-1}\geq 0.$$
By Proposition~\ref{p:decomposition}, $L_{s_j-1} \go_p H^{1,1}(X,\R)^{n-s_j}=0$ whenever $p \ge 2j+1$.
Observe also that
$$L_{s_j-1} \go_{2j} \go_q H^{1,1}(X,\R)^{n-s_j-1} = L_{s_j} \go_q H^{1,1}(X,\R)^{n-s_j-1} =0$$
whenever $q \ge 2j-1$ because $\go_{2j}$ is proportional to $M_{s_j}$ modulo $F_{s_j - 1}$
by Proposition~\ref{p:decomposition}.
Therefore, using the expansion \eqref{e:gm-omega}, we get from the last inequality
$$m^{-4j+2}L_{s_j-1} \bigg(2\binom{m}{2j}\binom{m}{2j-2}\go_{2j}\go_{2j-2}+\binom{m}{2j-1}^2\go_{2j-1}^2\bigg) c_1\cdots c_{n-s_j-1}+ o(1)\geq 0.$$
{Multiplying the last line by $(2j)!(2j-1)!$
and taking $m\to\infty$ give 
$$L_{s_j-1} \big((4j-2)\go_{2j}\go_{2j-2}+ (2j)\go_{2j-1}^2\big) c_1\cdots c_{n-s_j-1}\geq 0.$$
By Proposition~\ref{p:decomposition}, $\go_{2j}$ is positively proportional to $M_{s_j}$ modulo $F_{s_j - 1}$.
	Therefore, the last inequality implies that
$\go_{2j-1} \in W_{s_j}$.
\endproof

\begin{proof}[End of the proof of Theorem~\ref{t:main_1}]
	We can assume that $H_0\not=G$ as in Lemma \ref{l:max-length}.
	Take an integer $l$ as in that lemma. Since $l\leq n-2$, it is enough to show that
	$\ell(G/G_0)\leq l+1$.
	
	Since $H_0=G_0$ and $H_{l+1}=G$, in order to complete the proof, it is enough to show that $H_{i+1}/H_i$ is commutative.
	For this purpose, we can assume that $H_{i+1}\not=H_i$. Observe also that $H_{i+1}/H_i$ acts faithfully on $L_iH^{1,1}(X, \R)/\Nc^{i+1}(X)$
	and on $H^{1,1}(X,\R)/F_i$, see \eqref{e:Hi-Fi}. We continue to use the expansion \eqref{e:gm-omega}.
	By \eqref{e:Hi-Fi} and Proposition \ref{p:decomposition}, $h$ belongs to $H_i$ if and only if $\go_1(h)\in F_i$ for every K\"ahler class $\go$, or equivalently, $s_1(h)\leq i$. Therefore,
	for any $g \in H_{i+1}\setminus H_i$, we have $s_1(g)=i+1$ and, since $\go_1(g)\not\in F_{i+1}'$ by Proposition \ref{p:decomposition}, $H_{i+1}/H_i$ acts faithfully on $H^{1,1}(X,\R)/F_{i+1}'$.
	
	The following lemma is obvious.
	\begin{lemma} \label{l:commutative}
		Let $V$ be a vector space over an arbitrary field. Let $\Gamma$ be a subset of $\GL(V)$.
		Assume there is a vector subspace $W$ of $V$ invariant by $\Gamma$ such that $\Gamma$ acts trivially on $W$ and on $V/W$.
		Then $\Gamma$ is a commutative set.
	\end{lemma}

	By Lemmas~\ref{l:Sobolev} and Lemma~\ref{l:commutative}, it is enough to show that $g$ acts trivially on $W_{i+1}/F_{i+1}'$.
	Assume by contradiction that the last property is wrong. Then since
	$$(g^* - \Id)^2(H^{1,1}(X,\R)) \subset  F'_{s_1(g)} = F'_{i+1},$$
	there are classes $\ga$ and $M'$ in $W_{i+1}\setminus F_{i+1}'$ such that
	$(g^m)^*(\ga)=\ga+mM'$ modulo $F_{i+1}'$.
	Choose a class $\gb$ as in the definition of $W_{i+1}$.
	As we observed above, we can assume that $\gb$ is a K\"ahler class $\go$.
	Using that definition, the expansion \eqref{e:gm-omega}, Lemma~\ref{l:HR}, and Proposition \ref{p:decomposition}, we have for every K\"ahler classes $c_j$
	$$0 \le  -L_i (g^m)^*(\ga^2)c_1\cdots c_{n-i-2}\leq L_i M_{i+1} (g^m)^*(\go) c_1\cdots c_{n-i-2} = L_i  M_{i+1} \go_0  c_1\cdots c_{n-i-2}.$$
	Here, since  $L_i  M_{i+1}=L_{i+1}$ and $\go_j\in F_{i+1}$ for $j\geq 1$ (because $s_1=i+1$), these $\go_j$ do not contribute for last equality. Taking $m\to\infty$, we see
	that $L_i M'^2c_1\cdots c_{n-i-2}=0$.
	We have used here Lemma \ref{l:codim-1}(3) which implies that $L_iW_{i+1}F'_{i+1}\subset \Nc^{i+2}(X)$.	
	So $L_i M'^2=0$ modulo $\Nc^{i+2}(X)$. According to Lemma \ref{l:codim-1}, we have $M'\in F_{i+1}'$ which is a contradiction.
	This completes the proof of Theorem~\ref{t:main_1}.
\end{proof}

\section{Explicit examples} \label{s:examples}

Let $n \ge 1$ be an integer and let $E = \C /\gL$ be an elliptic curve.

\ssec{Torus examples for the optimality of the upper bound in Theorem~\ref{t:main-1map}}\label{ssec-exopt}
\hfill

Let $X \cnec E^n = \C^n /\gL^n$ with coordinates $z_1,\ldots,z_n$.
Define the automorphism $g : X \to X$ by
$$g(z_1,\ldots,z_n) = (z_1, z_1 + z_2, \ldots, z_{n-1} + z_n).$$
With respect to the bases $dz_1,\ldots,dz_n$ and $d\bar{z}_1,\ldots,d\bar{z}_n$ of $H^{1,0}(X)$ and $H^{0,1}(X)$ respectively,
both  $g^* : H^{1,0}(X) \to H^{1,0}(X)$ and
$g^* : H^{0,1}(X) \to H^{0,1}(X)$ correspond to the Jordan matrix
$$ J_n =
\begin{bmatrix}
1 &1          &            & \\
& \ddots & \ddots &     \\
&            & \ddots &  1 \\
& &   &   1
\end{bmatrix}.
$$
Since $g^* : H^{1,1}(X) \to H^{1,1}(X)$ is isomorphic to
the tensor product of $g^* : H^{1,0}(X) \to H^{1,0}(X)$
with $g^* : H^{0,1}(X) \to H^{0,1}(X)$,
the corresponding matrix of $g^* : H^{1,1}(X) \to H^{1,1}(X)$
is the Kronecker product (see e.g. \cite{Horn})
$$J_n \otimes J_n =
\begin{bmatrix}
J_n & J_n          &            & \\
& \ddots & \ddots &     \\
&            & \ddots &  J_n \\
& &   &   J_n
\end{bmatrix}.
$$
In particular, $g^* : H^{1,1}(X) \to H^{1,1}(X)$ is unipotent.
Set $N_n \cnec J_n - \Id_n$,
$$ N \cnec
\begin{bmatrix}
N_n &   0        &            & \\
& \ddots & \ddots &     \\
&            & \ddots &  0 \\
& &   &   N_n
\end{bmatrix} \ \ \text{ and } \ \ J \cnec
\begin{bmatrix}
0 &   J_n        &            & \\
& \ddots & \ddots &     \\
&            & \ddots &  J_n \\
& &   &   0
\end{bmatrix}.
$$
The linear map $(g^* - \Id) : H^{1,1}(X) \to H^{1,1}(X)$ is represented by the matrix $N + J$. We have $N^k = J^k = 0$ for every $k \ge n$ and
$$  N^{n-1} =
\begin{bmatrix}
N^{n-1}_n &   0        &            & \\
& \ddots & \ddots &     \\
&            & \ddots &  0 \\
& &   &   N^{n-1}_n
\end{bmatrix} \ \ \text{ and } \ \ J^{n-1} =
\begin{bmatrix}
0 &   \dots      &      0      & J^{n-1}_n   \\
& \ddots &  &  0   \\
&            & \ddots & \vdots \\
& &   &   0
\end{bmatrix}.
$$
Note that the multiplication of the first row of $N^{n-1}$ with the last column of $J^{n-1}$ is not zero. 
Therefore, since $NJ=JN$, we have
$$(N + J)^{2n-2} = \sum_{k = 0}^{2n-2} \binom{2n-2}{k}N^kJ^{2n-2-k} = \binom{2n-2}{n-1}N^{n-1}J^{n-1} \ne 0.$$
This example shows that
$$\|(g^m)^*: H^2(X,\C)\circlearrowleft\| \sim_{m \to \infty} Cm^{2(n-1)}$$
for some $C > 0$, thus the upper bound of the polynomial growth in
Theorem~\ref{t:main-1map} is optimal in terms of $n$.

\begin{remark}
More generally, we have
$$\|(g^m)^*:H^{p,q}(X,\R)\circlearrowleft\| \sim_{m \to \infty} Cm^{p(n-p) + q(n-q)}.$$ 
Thus the estimate in Theorem~\ref{t:main-1map} is optimal for $p,q\in \{0,1,n-1,n\}$.
We omit the proof.
\end{remark}

\ssec{Some explicit examples related to Conjecture~\ref{conj:main}}\label{ssec-exrelated}

\sssec{Torus examples}
\hfill

Denote by $U(n, \Z)$ the group of $n \times n$ upper triangular matrices whose entries are integers and whose diagonal entries are $1$. Let $E_{ij}$ be the $n\times n$-matrix whose $(i, j)$-entry is $1$ and other entries are $0$. It is well-known and easy to see directly that $U(n, \Z)$ is generated by
$$\tau_{ij} := I_n + E_{ij} \quad \text{with} \quad 1\le i<j\leq n\, ,$$
and $U(n, \Z)$ has nilpotency class $n-1$.

The group $G \cnec U(n, \Z)$ acts on $X = E^n$,
and the induced action on $H^2(E^{n}, \Z) = \wedge^{2} H^1(E^{n}, \Z)$ is
faithful and unipotent.
Hence, $G \subset \Aut(X)$ is a zero entropy subgroup.
We have $G \cap \Aut^0 (X) = \{1\}$ because $U(n, \Z)$ acts on $H^2(E^n, \Z)$ faithfully.
Thus, by Lemma \ref{l:unipotent_property}, we have
$$c_{{\rm ess}}(G, X) = c (U(n, \Z)) = n-1 = \dim X - 1.$$

\sssec{Some examples with positive Kodaira dimension}
\hfill

Let $n \ge 2$ and $\kappa$ be integers such that $1 \le \kappa \le n-1$.
Fix a smooth projective variety $B$ of dimension $\kappa$ having an ample canonical divisor $K_B$ 
(hence $\gk(B) = \dim B=\kappa$) 
such that there is a surjective morphism $\rho : B \to E$. One may find such a $B$ as a general member of a very ample linear system of $V := E^{\kappa + 1}$ and $\rho$ the projection to the first factor of $E^{\kappa+1}$.
Indeed, by the adjunction formula, $K_B = (K_V + B)|_B \sim B|_B$ is the restriction of an ample divisor, whence it is ample.

Consider the product
$X := E^{n-\kappa} \times B$. 
This is a smooth projective variety of dimension $n \ge 2$. 
By~\cite[Theorem 15.1]{Ueno},
	$$\gk(X) = \gk(B) = \dim B = \gk.$$}
The group $U(n-\kappa, \Z)$ acts faithfully on $X$ by
$$U(n-\kappa, \Z) = U(n-\kappa, \Z) \times \{\Id_B\} \,\, \subset\,\, \Aut((E, 0)^{n-\kappa}) \times \Aut(B) \,\, \subset\,\, \Aut(E^{n-\kappa} \times B)\, .$$
For each $i$ with $1 \le i \le n-\kappa$, we define $\rho_{i} \in \Aut (X)$ by
$$E^{n-\kappa} \times B \ni (z_1, \ldots, z_i, \ldots, z_{n-\kappa}, w) \mapsto (z_1, \ldots, z_i+ \rho(w), \ldots, z_{n-\kappa}, w) \in E^{n-\kappa} \times B\, .$$
Then $\rho_i \in \Aut (X/B)$ with respect to the canonical projection onto $B$.
Denote by $A$ the abelian subgroup of $\Aut(X)$ generated by the $\rho_i$ for $1\leq i\leq n-\kappa$.
Finally, define $G$ to be the subgroup of $\Aut(X)$ generated by $A$ and $U(n-\kappa,\Z)$.

Since $A$ (resp. $U(n-\kappa, \Z)$) acts on each fiber $E^{n-\kappa}$ as a translation (resp. zero entropy) group,
by the relative dynamical degree formula in \cite[Theorem 1.1]{DNT}, the action of $G$ on $X$ is of zero entropy. 
This can also be obtained using the definition of entropy and the fact that for $g\in G$ the Lipschitz norm of $g^m$ is bounded by a polynomial in $m$.

\begin{proposition} \label{prop52}
Let $X$ be as above. Then, $G_0=\{1\}$  and
	$$c_{{\rm ess}}(G, X) = \dim X - \kappa (X).$$
\end{proposition}

\begin{proof}
	
Using Blanchard's lemma~\cite[Cor.\,2.3]{Br11}, for $X = E^{n-\kappa} \times B$,  we obtain the following natural isomorphisms
$$\Aut^0 (X) \cong \Aut^0(E^{n-\kappa}) \times \Aut^0(B) \cong E^{n-\kappa} \times \{\Id_B \} \cong E^{n-\kappa} .$$
Here, we also use the assumption  that $B$ is of general type and hence $\Aut(B)$ is finite, see \cite[Cor.\,14.3]{Ueno}; it follows that $\Aut^0(B)$ is trivial.

Let $t \in B$ be a very general point. Then $\rho(t) \in (E, 0)$ is a non-torsion point, i.e., an infinite order element of the abelian group $(E, 0)$. Then, the action of $G$ on
$$q^{-1}(t)=E^{n-\kappa} \times \{t\} \simeq E^{n-\kappa}$$
is faithful and is given, for $z\in E^{n-\kappa}$, by
$$z\mapsto Az+\rho(t)b  \text{ \  with  \ } A\in U(n-\kappa,\Z) \text{ \ and \ } b\in\Z^{n-\kappa}\,.$$

If $t'\in B$ is another very general point, the intersection between $\rho(t)\Z^{n-\kappa}$ and $\rho(t')\Z^{n-\kappa}$ is trivial, i.e., equal to $\{0\}$. Therefore, we see that the identity is the only element in $G$ which belongs to $\Aut^0(X)$. Thus, we have $G_0=\{1\}$
and therefore $c_{{\rm ess}}(G, X) = c(G)$,
provided that there exists an embedding $G \hto \GL(V)$ of $G$ 
	as a unipotent subgroup for some finite dimensional vector space $V$.
Indeed, by the definition of $c_{{\rm ess}}(G, X)$,
there exists a finite index subgroup $G' \subset G$ such that 
$c_{{\rm ess}}(G, X) = c(G'/G'_0)$.
As $G_0=\{1\}$, we have $G'_0=\{1\}$. Thus 
$$c_{{\rm ess}}(G, X) = c(G'/G'_0) = c(G') = c(G),$$
where the last equality follows from Lemma~\ref{l:unipotent_property}.

It remains to find such an embedding $G \hto \GL(V)$ and show that $c(G) =  n-\kappa$.
This is now a purely group theoretical problem and we don't need to use the action on cohomology.
As $G$ acts faithfully on $q^{-1}(t)$, 
from the description of the action of $G$ on $q^{-1}(t)$
we may identify $G$ with the unipotent
affine transformation subgroup
$$P := \Z^{n-\kappa} \rtimes U({n-\kappa}, \Z)
\cong U(n -\kappa +1, \Z) \subset \GL(n -\kappa +1,\R).$$
Here, the last isomorphism is the natural one given by
$$\Z^{n-\kappa} \rtimes U({n-\kappa}, \Z) \ni f(x) = Ax + b \leftrightarrow
\begin{pmatrix}
A&b\\
0&1
\end{pmatrix} \in U(n-\kappa +1, \Z)\, .$$
It also follows that $c(G) = n-\kappa$.  This proves Proposition \ref{prop52}.
\end{proof}

\sssec{Some rationally connected examples}\label{eg:ratconn}
\hfill

Finally, we construct
examples with Kodaira dimension $-\infty$.
Let $n \ge 2$. Let $E_{\go}$ be an elliptic curve with period $\go = (-1 + \sqrt{-3})/2$ a primitive third root of unity. Let
$$\overline{X}_n := E_{\go}^{n}/\langle -\go I_n \rangle ,$$
$\pi: E_{\go}^{n} \to \overline{X}_n$ the quotient map, and $X_n \to \overline{X}_n$ the blow-up along the maximal ideals of all singular points of $\overline{X}_n$.
Then $X_n$ is a smooth projective variety and the action of $G \cnec U(n, \Z)$ on $E_{\go}^{n}$ descends to a faithful {\it biholomorphic} action on
$X_n$.
As the $G$-action on $E_{\go}^{n}$ has zero entropy,
	so do the $G$-action on $X_n$.
We verify that
$$c_\ess(G,X_n) = n-1.$$
 Furthermore,
 if $n \in \{2, 3, 4, 5\}$, then
 $X_n$ is rationally connected
 (see \emph{e.g.} \cite[Proof of Corollary \,4.6]{Og}
 and also \cite{OT}),
 and hence
 $\kappa(X_n) = -\infty$.

In view of
the above examples
we ask the following.

\begin{question}\label{ques51}
	Let $n \ge 6$.
	Can we construct an $n$-dimensional
	smooth projective variety $X$ of Kodaira dimension $-\infty$
	(e.g. a rational, or rationally connected variety) which admits
	a zero entropy subgroup $G$ such that $c_\ess(G,X) = n - 1$?
\end{question}

\small

\end{document}